\documentclass{amsart}
\usepackage{amsfonts}
\usepackage{amsmath}
\usepackage{amssymb}
\usepackage{graphicx}

\theoremstyle{definition}
\newtheorem{definition}{Definition}[section]
\theoremstyle{plain}
\newtheorem{lemma}[definition]{Lemma}
\newtheorem{proposition}[definition]{Proposition}
\newtheorem{theorem}[definition]{Theorem}

\numberwithin{equation}{section}

\begin{document}
\title[Variational Inequalities]{Variational Inequalities on Geodesic Spaces}
\author[E. Hacioglu]{Emirhan Hacioglu}
\address[E. HACIO\u{G}LU]{Department of Mathematics, Yildiz Technical
University, Davutpasa Campus, Esenler, 34220 Istanbul, Turkey}
\email{\texttt{emirhanhacioglu@hotmail.com}}
\author[V. Karakaya]{Vatan Karakaya}
\address[V. Karakaya]{Department of Mathematical Engineering, Yildiz
Technical University, Davutpasa Campus, Esenler, 34210 Istanbul,Turkey}
\email{\texttt{vkkaya@yahoo.com}}

\begin{abstract}
In this paper, we introduce a new variational inequality problem(VIP)
associated with nonself multivalued nonexpansive mappings in $CAT(0)$
spaces.
\end{abstract}

\maketitle

\section{Introduction and Preliminaries}

\qquad\ Let $(X,d)$ be a metric space then the family of nonempty, closed
and convex subsets of $X$, the family of nonempty compact and convex subsets
of $X,$ the family of nonempty compact subsets of $X,$ the family of
nonempty closed and bounded convex subsets of $X$ will be denoted by $%
C(X),KC(X),$ $K(X),$ $CB(X)$, respectively. Let $H$ be a Haussdorf Metric on 
$CB(X),$ defined by

\begin{equation*}
H(A,B)=\max \{\sup_{x\in A}d(x,B),\sup_{x\in B}d(x,A)\}
\end{equation*}%
\ where $d(x,B)=\inf \{d(x,y);y\in B\}.$ A multivalued mappings $%
T:X\rightarrow 2^{X}$ is called nonexpansive if for all $x,y\in X$ \ \ 
\begin{equation*}
H(Tx,Ty)\leq d(x,y)
\end{equation*}%
is satisfied. A point is called fixed point of $T$ \ if $x\in Tx$ and the
set of all fixed points of $T$ is denoted by $F(T).$Many iterative processes
to find a fixed point of multivalued mappings have been introduced in metric
spaces and Banach spaces. One of them is defined by Nadler\cite{nadler} as
generalization of Picard as follows; 
\begin{equation*}
x_{n+1}\in Tx_{n}.
\end{equation*}%
A multivalued version of Mann and Ishikawa fixed point procedures goes as
follow; 
\begin{equation*}
x_{n+1}\in (1-\alpha _{n})x_{n}+\alpha _{n}Tx_{n}
\end{equation*}%
and%
\begin{eqnarray*}
x_{n+1} &\in &(1-\alpha _{n})x_{n}+\alpha _{n}Ty_{n}, \\
y_{n} &\in &(1-\beta _{n})x_{n}+\beta _{n}Tx_{n}
\end{eqnarray*}%
where $\{\alpha _{n}\}$ and $\{\beta _{n}\}$ are sequences\ in$\ [0,1]$.

Gursoy and Karakaya \cite{gurs} (see also \cite{gurs2}) introduced Picard-S
iteration as follows;%
\begin{eqnarray*}
x_{n+1} &=&Ty_{n}, \\
y_{n} &=&(1-\alpha _{n})Tx_{n}+\alpha _{n}Tz_{n}, \\
z_{n} &=&(1-\beta _{n})x_{n}+\beta _{n}Tx_{n}
\end{eqnarray*}%
where $\{\alpha _{n}\}$ and $\{\beta _{n}\}$ are$\ $sequences\ in$\ [0,1]$.
They have showed that it converges to fixed point of contraction mappings
faster than Ishikawa, Noor, SP, CR, S and some other iterations. Also they
use it to solve differential equations. Now, we define multivalued version
of Picad-S iteration in $CAT(0)$ spaces as follows; Let $K\subset CAT(0)$ be
a nonempty, closed and convex subset, $T:K\rightarrow C(K)$ is a mapping, $%
x_{0}\in K$. then for any $n\geq 0$, the proximal multivalued Picard-S
iteration is defined by 
\begin{eqnarray}
x_{n+1} &=&P_{K}(u_{n}),  \label{1.1} \\
y_{n} &=&P_{K}((1-\alpha _{n})w_{n}\oplus \alpha _{n}v_{n}),  \notag \\
z_{n} &=&P_{K}((1-\beta _{n})x_{n}\oplus \beta _{n}w_{n})  \notag
\end{eqnarray}%
where $P_{K}$ is a metric projection, $\{\alpha _{n}\}$ and $\{\beta _{n}\}$
are sequences in$\ [0,1]$ with $\liminf_{n}(1-\beta _{n})\beta _{n}>0$, $%
u_{n}\in Ty_{n},$ $v_{n}\in Tz_{n}$ and $w_{n}\in Tx_{n}$.

\qquad Before the results we give some definitions and lemmas about $CAT(0)$
and \ $\Delta -$convergences.

\qquad Let $(X,d)$ be a metric space, $x,y\in X$ and $C\subseteq X$ nonempty
subset. A geodesic path (or shortly a geodesic) joining $x$ and $y$ is a map 
$c:[0,t]\subseteq 
%TCIMACRO{\U{211d} }%
%BeginExpansion
\mathbb{R}
%EndExpansion
\rightarrow X$ \ such that $c(0)=x$, $c(t)=y$ and $d(c(r),c(s))=|r-s|$ for
all $r,s\in \lbrack 0,t]$. In particular $c$ is an isometry and $%
d(c(0),c(t))=t.$ The image of $c,$ $c([0,t])$ is called geodesic segment
from $x$ to $y$ and it is unique (it not necessarily be unique) then it is
denoted by $[x,y].$ $z\in \lbrack x,y]$ if and only if for an $\lambda \in
\lbrack 0,1]$ such that $d(z,x)=(1-\lambda )d(x,y)$ and $d(z,y)=\lambda
d(x,y).$ The point $z$ is denoted by $z=(1-\lambda )x\oplus \lambda y.$ If
for every $x,y\in X$ there is a geodesic path then $(X,d)$ called geodesic
space and uniquely geodesic space if that geodesic path is unique for any
pair $x,y.$ A subset $C\subseteq X$ is called convex if it contains all
geodesic segment joining any pair of points in it.

\qquad In geodesic metric space $(X,d)$, a geodesic triangle $\Delta (x,y,z)$
\ consist of three point $x,y,z$ as vertices and three geodesic segments of
any pair of these points, that is, $q\in $ $\Delta (x,y,z)$ means that $q\in
\lbrack x,y]\cup \lbrack x,z]\cup \lbrack y,z].$ The triangle $\overline{%
\Delta }(\overline{x},\overline{y},\overline{z})$ in $(%
%TCIMACRO{\U{211d} }%
%BeginExpansion
\mathbb{R}
%EndExpansion
^{2},d_{2})$ is called comparison triangle for the triangle $\Delta (x,y,z)$
such that $d(x,y)=d_{2}(\overline{x},\overline{y}),d(x,z)=d_{2}(\overline{x},%
\overline{z})$ and $d(y,z)=d_{2}(\overline{y},\overline{z})$ \ A point point 
$\overline{z}\in \lbrack \overline{x},\overline{y}]$ called comparison point
for $z\in \lbrack x,y]$ if $d(x,z)=d_{2}(\overline{x},\overline{z}).$ A
geodesic triangle $\Delta (x,y,z)$ in $X$ is satisfied $CAT(0)$ inequality if%
$\ d(p,q)\leq d_{2}(\overline{p},\overline{q})$ for all $p,q\in $ $\Delta
(x,y,z)$ where $\overline{p},\overline{q}\in \overline{\Delta }(\overline{x},%
\overline{y},\overline{z})$ are the comparison points of $p,q$ respectively.
A geodesic space is called $CAT(0)$ space if for all geodesic triangles
satisfies $CAT(0)$ inequality or alternatively: A geodesic space is called $%
CAT(0)$ space if and only if the inequality 
\begin{equation*}
d^{2}(x,(1-\lambda )y\oplus \lambda z)\leq (1-\lambda )d^{2}(x,y)+\lambda
d^{2}(x,z)-\frac{R}{2}\lambda (1-\lambda )d^{2}(y,z),
\end{equation*}%
satisfied for every $x,y,z\in X$,\ $\lambda \in \lbrack 0,1]$.

\begin{proposition}
\cite{dogh2}Let $(X,d)$ be a $CAT(0)$ space Then, for any $x,y,z\in X$ and $%
\lambda \in \lbrack 0,1]$, we have%
\begin{equation*}
d((1-\lambda )x\oplus \lambda y,z)\leq (1-\lambda )d(x,z)+\lambda d(y,z).
\end{equation*}
\end{proposition}

Let $\{x_{n}\}$ be a bounded sequence on $X$ and $x\in X$. $\ $Then, with
setting%
\begin{equation*}
r(x,\{x_{n}\})=\limsup_{n\rightarrow \infty }d(x,x_{n})
\end{equation*}%
the asymptotic radius of $\{x_{n}\}$ is defined by%
\begin{equation*}
r(\{x_{n}\})=\inf \{r(x,\{x_{n}\});x\in X.\},
\end{equation*}%
the asymptotic radius of $\{x_{n}\}$ with respect to $K\subseteq X$ is
defined by%
\begin{equation*}
r_{K}(\{x_{n}\})=\inf \{r(x,\{x_{n}\});x\in K.\}
\end{equation*}%
and the asymptotic center of $\{x_{n}\}$ is defined by%
\begin{equation*}
A(\{x_{n}\})=\{x\in X:r(x,\{x_{n}\})=r(\{x_{n}\})\}.
\end{equation*}%
and let $\omega _{w}(x_{n}):=\cup A(\{x_{n}\})$ where union is taken on all
subsequences of $\{x_{n}\}.$

\begin{definition}
\label{(def2.8)}\cite{esp}A sequence $\{x_{n}\}\subset X$ \ is said to be $%
\Delta -$ convergent to $x\in X$ if \ $x$ is the unique asymptotic center of
\ all subsequence $\{u_{n}\}$ of $\{x_{n}\}$, i.e.$\omega _{w}(x_{n}):=\cup
A(\{x_{n}\})$ $=\{x\}$ . In this case we write $\Delta -\lim_{n}x_{n}=x$.
\end{definition}

\begin{lemma}
\label{(lemma2.10)}\cite{dogh2}
\end{lemma}

\begin{enumerate}
\item[i)] \textit{Every bounded sequence in a }complete $CAT(0)$ space%
\textit{\ has a }$\Delta $\textit{-convergent subsequence}

\item[ii)] \textit{If }$K$\textit{\ is a closed convex subset of a }complete 
$CAT(0)$\textit{\ and if }$\{x_{n}\}$\textit{\ is a bounded sequence in }$K$%
\textit{, then the\ asymptotic center of }$\{x_{n}\}$\textit{\ is in }$K$
\end{enumerate}

\begin{lemma}
\label{(lemma2.11)}\cite{dogh2} If $\{x_{n}\}$ is a bounded sequence in $X$
with $A(\{x_{n}\})=\{x\}$ and $\{u_{n}\}$ is a subsequence of $\{x_{n}\}$
with $A(\{u_{n}\})=u$ and the sequence $\{d(x_{n},u)\}$ converges, then $x=u$
\end{lemma}

\begin{theorem}
\cite{shim}\label{the1}Let $X$ be a bounded, complete and uniformly convex
metric space. If $T$ is a multivalued nonexpansive mapping which assigns to
each point of $X$ a nonempty compact subset of $X$, then $T$ has a fixed
point in $X$.
\end{theorem}

In a complete $CAT(0)$ space, the metric projection $P_{K}(x)$ of $x$ onto a
nonempty, closed and convex subset $K$ is singleton and nonexpansive.

The concept of inner-product has been generalized from Hilbert space to a $%
CAT(0)$ space $X$ by Berg and Nikolaev \cite{berg}. as follows: For any $%
a,b\in X,$ with denoting $\overrightarrow{ab}$ as a vector in $X$,
quasi-linearization mapping defined as%
\begin{eqnarray*}
\langle ,\rangle &:&(X\times X)\times (X\times X)\rightarrow 
%TCIMACRO{\U{211d} }%
%BeginExpansion
\mathbb{R}
%EndExpansion
, \\
\langle \overrightarrow{ab},\overrightarrow{cd}\rangle &=&\frac{1}{2}%
[d^{2}(a,d)+d^{2}(b,c)-d^{2}(a,c)-d^{2}(b,d)]
\end{eqnarray*}%
for all $a,b,c,d\in X$ and satisfies following properties%
\begin{eqnarray*}
\langle \overrightarrow{ab},\overrightarrow{ab}\rangle &=&d^{2}(a,b) \\
\langle \overrightarrow{ab},\overrightarrow{cd}\rangle &=&-\langle 
\overrightarrow{ba},\overrightarrow{cd}\rangle \\
\langle \overrightarrow{ab},\overrightarrow{ab}\rangle &=&\langle 
\overrightarrow{ae},\overrightarrow{cd}\rangle +\langle \overrightarrow{eb},%
\overrightarrow{cd}\rangle \\
\langle \overrightarrow{ab},\overrightarrow{cd}\rangle &=&d(a,b)d(c,d)
\end{eqnarray*}%
for all $a,b,c,d,e\in X$ The last properties is known as a Cauchy-Schwarz
inequality and it is a characterization of $CAT(0)$ space: A geodesic metric
space is a $CAT(0)$ if and only if it satisfies Cauchy-Schwarz inequality.\ 

\begin{lemma}
\cite{berg}\label{lem1}Let $X$ be a $CAT(0)$ and $K$ be a nonempty and
convex subset of $X$, $x\in X$ and $u\in K$. Then $u=P_{K}(x)$ if and only
if 
\begin{equation*}
\langle \overrightarrow{xu},\overrightarrow{yu}\rangle \leq \text{ for all }%
y\in K
\end{equation*}
\end{lemma}

Let $X$ be a real Hilbert space and $K\subset X$ \ be nonempty closed and
convex. A operator $A:K\rightarrow 2^{X}$ is called monotone if and only if 
\begin{equation*}
\langle x-y,x^{\ast }-y^{\ast }\rangle \geq 0
\end{equation*}%
for all $x,y\in X,$ $x^{\ast }\in Ax,y^{\ast }\in Ay$. If $A$ is a monotone
operator then the variational inequality associated with $A$ is finding $%
(u,x)_{u\in Ax}$ such that 
\begin{equation*}
\langle u,y-x\rangle \geq 0\text{, for all }y\in K\text{ }
\end{equation*}%
The VIPs associated with monotone operators have applications in applied
mathematics. For interested readers can find more informations about VIPs
and their applications in the book by Kinderlehrer and Stampacchia (see 
\cite{kind1, kind2}
).

Now let $X$ be a complete $CAT(0)$ space, $K\subset X$ be nonempty, closed
and convex and $T:K\rightarrow X$ be a nonexpansive mapping. In 2015,
Khatibzadeh, \& Ranjbar \cite{kha} defined the variational inequality
associated with the nonexpansive mapping $T$ \ as follows

\begin{equation*}
\text{Find }x\in K\text{ such that }\langle \overrightarrow{Txx},%
\overrightarrow{xy}\rangle \geq \text{ for all }y\in K
\end{equation*}%
They prove some existence and convergence results for this problem.

In this paper, we define variational inequality associated with the a
non-self multivalued nonexpansive mapping $T:K\rightarrow KC(X)$ as follows%
\begin{equation}
\text{Find }(u,x)_{u\in Tx}\text{ such that }\langle \overrightarrow{ux},%
\overrightarrow{xy}\rangle \geq 0\text{ for all }y\in K  \label{1.2}
\end{equation}%
and we prove some existence and convergence theorems for this problem.

\section{Existence of A Solution}

In this section, it is assumed that $X$ is a complete $CAT(0)$ and $K$ is a
nonempty, closed and convex subset of $X$.

\begin{definition}
If $K$ is also bounded subset of $X$ and $T:K\rightarrow C(X)$. Then the
projection $P_{K}T$ of multivalued mapping $T$ onto $K$ is defined by%
\begin{eqnarray*}
P_{K}^{\ast }T(x) &=&\bigcup\limits_{x^{\prime }\in Tx}\{P_{K}(x^{\prime })\}
\\
&=&\{P_{K}(x^{\prime }):x^{\prime }\in Tx\} \\
&=&\{v\in K:d(x^{\prime },v)=D(x^{\prime },K),\text{ }x^{\prime }\in Tx\}
\end{eqnarray*}%
where $P_{K}$ is metric projection and $D(x^{\prime },K)=\inf_{v^{\prime
}\in K}d(x^{\prime },v^{\prime }).$
\end{definition}

\begin{lemma}
$P_{K}^{\ast }T(x)$ is multivalued nonexpansive mapping from $K$ to $2^{K}$
\end{lemma}

\begin{proof}
Since $K$ is closed, convex and bounded, $P_{K}^{\ast }(Tx)\subset K$. We
also have 
\begin{eqnarray*}
H(P_{K}^{\ast }(Tx),P_{K}^{\ast }(Ty)) &=&\max
\{\sup\limits_{P_{K}(x^{\prime })\in P_{K}^{\ast
}Tx}\inf\limits_{P_{K}(y^{\prime })\in P_{K}^{\ast }Ty}d(P_{K}(x^{\prime
}),P_{K}(y^{\prime }), \\
&&\sup\limits_{P_{K}(y^{\prime })\in P_{K}^{\ast
}Ty}\inf\limits_{P_{K}(x^{\prime })\in P_{K}^{\ast }Tx}d(P_{K}(y^{\prime
}),P_{K}(x^{\prime })\} \\
&\leq &\max \{\sup\limits_{x^{\prime }\in Tx}\inf\limits_{y^{\prime }\in
Ty}d(x^{\prime },y^{\prime }),\sup\limits_{y^{\prime }\in
Ty}\inf\limits_{x^{\prime }\in Tx}d(y^{\prime },x^{\prime }\} \\
&=&H(Tx,Ty) \\
&\leq &d(x,y)\text{.}
\end{eqnarray*}

by the nonexpansiveness of $P_{K}$.\qquad
\end{proof}

\begin{lemma}
If $T$ is compact valued then $P_{K}^{\ast }$ is compact valued.
\end{lemma}

\begin{proof}
Let $(v_{n})\subset P_{K}^{\ast }T(x)$ be a sequence then there is a
sequence $(x_{n}^{\prime })\subset Tx$ such that for all $n\in 
%TCIMACRO{\U{2115} }%
%BeginExpansion
\mathbb{N}
%EndExpansion
,$ there $v_{n}=P_{K}(x_{n}^{\prime }).$Since $T$ have compact values then $%
(x_{n}^{\prime })$ have convergent subsequence $(x_{n_{k}}^{\prime })$ with $%
\lim_{k\rightarrow \infty }x_{n_{k}}^{\prime }=z\in Tx$ and since for all $%
k\in 
%TCIMACRO{\U{2115} }%
%BeginExpansion
\mathbb{N}
%EndExpansion
,$%
\begin{equation*}
d(P_{K}(x_{n_{k}}^{\prime }),P_{K}(z))\leq d(x_{n_{k}}^{\prime },z)
\end{equation*}%
we get that the sequence $(v_{n})=(P_{K}(x_{n}^{\prime }))$ have convergent
subsequence $(v_{n_{k}})=(P_{K}(x_{n_{k}}^{\prime }))$ with $%
\lim_{k\rightarrow \infty }(P_{K}(x_{n_{k}}^{\prime }))=P_{K}(z)\in
P_{K}^{\ast }T(x)$ therefore $P_{K}^{\ast }Tx$ is compact.
\end{proof}

\begin{theorem}
\label{teo2}If $T:K\rightarrow KC(X)$. Then there exists a solution $%
(u,x)_{u\in Tx}$ of the variational inequality (\ref{1.2})
\end{theorem}

\begin{proof}
Since $X$ is uniformly convex and $T$ is compact valued, $P_{K}T$ have fixed
point $p\in P_{K}T(p)\subset K$ by Theorem \ref{the1}. There exist $%
p^{\prime }\in Tp$ such that $p=P_{K}(p^{\prime })$ by definition of $%
P_{K}T. $ we have $\langle p^{\prime }p,yp\rangle \leq 0$ for all $y\in K$
by Lemma \ref{lem1}. Hence we have 
\begin{equation*}
\langle p^{\prime }p,yp\rangle \geq 0\text{ for all }y\in K
\end{equation*}%
where \ $p^{\prime }\in Tp$, that is $(p^{\prime },p)_{p^{\prime }\in Tp}$
is a solution of the problem (\ref{1.2}).
\end{proof}

\begin{theorem}
\label{teo3}If $x\in int(K)$ and $(u,x)_{u\in Tx}$ is a solution of problem (%
\ref{1.2}) then $x\in F(T)$, i.e., $u=x.$
\end{theorem}

\begin{proof}
There exists $\epsilon >0$ such that $B(x,\epsilon )\subset K.$ Let take $%
t\in (0,1)$ such that $tx\oplus (1-t)u\in B(x,\epsilon ),$ that is, $%
d(x,tx\oplus (1-t)u)=(1-t)d(x,u)<\epsilon $ . Since $B(x,\epsilon )\subset K$
then $tx\oplus (1-t)u\in K$ and $d(u,tx\oplus (1-t)u)=td(u,x)$ so we have%
\begin{eqnarray*}
0 &\leq &2\langle \overrightarrow{ux},\overrightarrow{x(tx\oplus (1-t)u)}%
\rangle \\
&=&d^{2}(u,tx\oplus (1-t)u)-d^{2}(x,u)-d^{2}(x,tx\oplus (1-t)u) \\
&=&t^{2}d^{2}(x,u)-d^{2}(x,u)-(1-t)^{2}d^{2}(x,u) \\
&=&2(t^{2}-1)d(x,u)\leq 0.
\end{eqnarray*}%
and which implies%
\begin{equation*}
2(t-1)d(x,u)=0
\end{equation*}%
since $t\in (0,1)$ then $d(x,u)=0$. Hence $u=x\in Tx$
\end{proof}

If $K$ is not bounded, the problem (\ref{1.2}) does not always have a
solution. However if $o\in X$ be arbitrary and setting $K_{r}=K\cap B(o,r)$
then if\ $K_{r}\neq \emptyset $ By Theorem \ref{teo2} there is $x_{r}\in
K_{r}$ \ such that $(u_{r},x_{r})_{u_{r}\in Tx}$ is a solution of problem%
\begin{equation}
\langle \overrightarrow{u_{r}x_{r}},\overrightarrow{x_{r}y}\rangle \geq 0%
\text{ for all }y\in K_{r}  \label{2.1}
\end{equation}

\begin{theorem}
\label{teo4}The problem (\ref{1.2}) have a solution if and only if there is
a $r>0$ such that the solution of the problem (\ref{2.1}) $%
(u_{r},x_{r})_{u_{r}\in Tx_{r}},x_{r}\in K_{r}$ satisfies $d(o,x_{r})<r.$
\end{theorem}

\begin{proof}
If the problem \ref{1.2} have a solution $(u,x)_{x\in Tx}$ then $(u,x)_{x\in
Tx}$ is a solution of the problem (\ref{2.1}) and $d(o,x)<r$ is satisfied.
Now, let there is a $r>0$ such that the solution of the problem (\ref{2.1}) $%
(u_{r},x_{r})_{u_{r}\in Tx_{r}}$, $x_{r}\in K_{r}$ satisfies $d(o,x_{r})<r$
and $y\in K$ be arbitrary. Then we can chose $t\in (0,1)$ such that $%
(1-t)x_{r}\oplus ty\in B(o,r),$ that is, $(1-t)x_{r}\oplus ty\subset K_{r}$
and $d(x_{r},(1-t)x_{r}\oplus ty)=td(x_{r},y)$. Then%
\begin{eqnarray*}
0 &\leq &2\langle \overrightarrow{u_{r}x_{r}},\overrightarrow{x_{r}y}\rangle
\\
&=&d^{2}(u_{r},(1-t)x_{r}\oplus
ty)-d^{2}(x_{r},u_{r})-d^{2}(x_{r},(1-t)x_{r}\oplus ty) \\
&\leq
&(1-t)d^{2}(u_{r},x_{r})+td^{2}(u_{r},y)-t(1-t)d^{2}(x_{r},y)-d^{2}(x_{r},u_{r})-t^{2}d^{2}(x_{r},y)
\\
&=&2t(d^{2}(u_{r},y)+d^{2}(x_{r},x_{r})-d^{2}(u_{r},x_{r})-d^{2}(x_{r},y)) \\
&=&2t\langle \overrightarrow{u_{r}x_{r}},\overrightarrow{x_{r}y}\rangle 
\text{.}
\end{eqnarray*}%
Hence%
\begin{equation*}
\langle \overrightarrow{u_{r}x_{r}},\overrightarrow{x_{r}y}\rangle \geq 0%
\text{, for all }y\in K\text{ }
\end{equation*}%
that is $(u_{r},x_{r})_{u_{r}\in Tx_{r}}$ is a solution of the problem (\ref%
{1.2}) .
\end{proof}

\begin{theorem}
Let $T:K\rightarrow KC(X)$ and $o\in X$ be fixed. If there exist $x_{0}\in K$
and $u_{0}\in Tx_{0}$ such that 
\begin{equation*}
\frac{\langle \overrightarrow{ux},\overrightarrow{x_{0}x}\rangle -\langle 
\overrightarrow{u_{0}x_{0}},\overrightarrow{x_{0}x}\rangle }{d(x,x_{0})}%
\rightarrow \infty \text{ as }d(x,o)\rightarrow \infty
\end{equation*}%
where $u\in Tx$ \ such that $d(x,u)=d(x,Tx)$ then the problem (\ref{1.2})
have a solution.
\end{theorem}

\begin{proof}
Let $R,M\in 
%TCIMACRO{\U{211d} }%
%BeginExpansion
\mathbb{R}
%EndExpansion
$ such that $d(u_{0},u_{0})<M$, $d(x_{0},o)<r$ \ and 
\begin{equation*}
\langle \overrightarrow{ux},\overrightarrow{x_{0}x}\rangle -\langle 
\overrightarrow{u_{0}x_{0}},\overrightarrow{x_{0}x}\rangle \geq
Md(u_{0},x_{0})
\end{equation*}%
for all $x\in K,$ $d(x,o)\geq r$. Then%
\begin{eqnarray*}
\langle \overrightarrow{ux},\overrightarrow{x_{0}x}\rangle &\geq &\langle 
\overrightarrow{u_{0}x_{0}},\overrightarrow{x_{0}x}\rangle +Md(u_{0},x_{0})
\\
&\geq &-d(u_{0},x_{0})d(x_{0},x)+Md(u_{0},x_{0}) \\
&\geq &(M-d(x_{0},x))d(u_{0},x_{0}) \\
&\geq &(M-d(x_{0},x))(d(x,o)-d(x_{0},o))
\end{eqnarray*}%
for $r=d(x,o)$. If $(u_{r},x_{r})_{u_{r}\in Tx_{r}}$ is a solution of the
problem (\ref{2.1}) then since%
\begin{equation*}
\langle \overrightarrow{u_{r}x_{r}},\overrightarrow{x_{0}x_{r}}\rangle
=-\langle \overrightarrow{u_{r}x_{r}},\overrightarrow{x_{r}x_{0}}\rangle
\leq 0
\end{equation*}%
holds so we have $d(x_{r},o)<r$. Hence by Theorem \ref{teo4} the problem (%
\ref{1.2}) has a solution.
\end{proof}

\section{Convergence Results to The Solutions}

In this section, it is assumed that $X$ is a complete $CAT(0)$ and $K$ is a
nonempty, closed and convex subset of $X$.

\begin{theorem}
\label{(teo4.3)}If $T:K\rightarrow KC(X)$ is a nonexpansive mapping and $%
\{x_{n}\}$ is a bounded sequence in $K$ with $\Delta -\lim_{n\rightarrow
\infty }x_{n}=z$ and $\lim_{n\rightarrow \infty }d(x_{n},Tx_{n})=0$ then $%
z\in K$ and $z\in T(z).$
\end{theorem}

\begin{proof}
By Lemma \ref{(lemma2.10)}, $z\in K.$We can find a sequence $\{y_{n}\}$ such
that $y_{n}\in Tx_{n},$ $d(x_{n},y_{n})=d(x_{n},Tx_{n}),$ so we have $%
\lim_{n\rightarrow \infty }d(x_{n},y_{n})=0$ and we can find a sequence $%
\{z_{n}\}$ in $Tz$ such that $d(y_{n},z_{n})=d(y_{n},Tz).$ Then Since $Tz$
is compact, there is a convergent subsequence $\{z_{n_{i}}\}$ of $\{z_{n}\},$
say $\lim_{i\rightarrow \infty }z_{n_{i}}=u\in Tz$.%
\begin{eqnarray*}
d(x_{n_{i}},u) &\leq
&d(x_{n_{i}},y_{n_{i}})+d(y_{n_{i}},z_{n_{i}})+d(z_{n_{i}},u) \\
&\leq &d(x_{n_{i}},y_{n_{i}})+d(y_{n_{i}},Tz)+d(z_{n_{i}},u) \\
&\leq &d(x_{n_{i}},y_{n_{i}})+H(Tx_{n_{i}},Tz)+d(z_{n_{i}},u) \\
&\leq &d(x_{n_{i}},y_{n_{i}})+H(Tx_{n_{i}},Tz)+d(z_{n_{i}},u)
\end{eqnarray*}%
implies that \ $\limsup_{i\rightarrow \infty }d(x_{n_{i}},u)\leq
\limsup_{i\rightarrow \infty }H(Tx_{n_{i}},Tz)$ and $\Delta
-\lim_{i\rightarrow \infty }x_{n_{i}}=z$ Because of $T$ is multivalued
nonexpansive mapping,%
\begin{equation*}
H(Tx_{n_{i}},Tz)\leq d(x_{n_{i}},z)
\end{equation*}%
which implies that 
\begin{equation*}
\limsup_{i\rightarrow \infty }d(x_{n_{i}},u)\leq \limsup_{i\rightarrow
\infty }H^{2}(Tx_{n_{i}},Tz)\leq \limsup_{i\rightarrow \infty
}d^{2}(x_{n_{i}},z)
\end{equation*}%
which implies that $z=u\in Tz.$
\end{proof}

\begin{lemma}
\label{(lemma4.3)}If $T:K\rightarrow KC(X)$ is a nonexpansive mapping and $%
\{x_{n}\}$ is a bounded sequence in $K$ with $\lim_{n\rightarrow \infty
}d(x_{n},Tx_{n})=0$ and $\{d(x_{n},p)\}$ converges for all $p\in F(T)$ then $%
\omega _{w}(x_{n})\subseteq F(T)$ and $\omega _{w}(x_{n})$ include exactly
one point.
\end{lemma}

\begin{proof}
Let take $u\in \omega _{w}(x_{n})$ then there exist subsequence $\{u_{n}\}$
of $\{x_{n}\}$ with $A(\{u_{n}\})=\{u\}.$Then by Lemma \ref{(lemma2.10)}
there exist subsequence $\{v_{n}\}$ of $\{u_{n}\}$ with $\Delta
-\lim_{n\rightarrow \infty }v_{n}=v\in K$ . Then by Theorem \ref{(teo4.3)}
we have $v\in F(T)$ and by Lemma \ref{(lemma2.11)} we conclude that $u=v,$
hence we get $\omega _{w}(x_{n})\subseteq F(T)$. Let take subsequence $%
\{u_{n}\}$ of $\{x_{n}\}$with $A(\{u_{n}\})=\{u\}$ and $A(\{x_{n}\})=\{x\}.$
Because of $v\in \omega _{w}(x_{n})\subseteq F(T)$, $\{d(x_{n},u)\}$
converges, so by Lemma \ref{(lemma2.11)} we have $x=u,$ this means that $%
\omega _{w}(x_{n})$ include exactly one point.
\end{proof}

\begin{theorem}
\label{teo4.4}If $T:K\rightarrow C(X)$ is a nonexpansive mapping with $%
F(T)\neq \emptyset $ and $Tp=\{p\}$ for all $p\in F(T)$ and $\{x_{n}\}$ is a
sequence in $K$ \ defined by (\ref{1.1}) with $\lim \inf_{n\rightarrow
\infty }\beta _{n}(1-\beta _{n})>0$ then $\{x_{n}\}$ is bounded, $%
\lim_{n\rightarrow \infty }d(x_{n},Tx_{n})=0$ and $\{d(x_{n},p)\}$ converges
for all $p\in F(T).$
\end{theorem}

\begin{proof}
Let $p\in F(T)$ then for any $x\in K,$we have that%
\begin{equation*}
d(Tx,p)\leq H(Tx,Tp)\leq d(x,p)
\end{equation*}%
since metric projection $P_{K}$ is nonexpansive and $P_{K}(p)=\{x\in
K:d(p,x)=d(p,K)\}=\{p\}$ we have%
\begin{eqnarray*}
d^{2}(y_{n},p) &=&d^{2}(P_{K}((1-\beta _{n})x_{n}\oplus \beta
_{n}v_{n}),P_{K}(p)) \\
&\leq &d^{2}((1-\beta _{n})x_{n}\oplus \beta _{n}v_{n},p) \\
&\leq &(1-\beta _{n})d^{2}(x_{n},p)+\beta _{n}d^{2}(v_{n},p) \\
&&-(1-\beta _{n})\beta _{n}d^{2}(x_{n},v_{n}) \\
&\leq &(1-\beta _{n})d^{2}(x_{n},p)+\beta _{n}d^{2}(v_{n},Tp) \\
&&-(1-\beta _{n})\beta _{n}d^{2}(x_{n},Tx_{n}) \\
&\leq &(1-\beta _{n})d^{2}(x_{n},p)+\beta _{n}H^{2}(Tx_{n},Tp) \\
&&-(1-\beta _{n})\beta _{n}d^{2}(x_{n},Tx_{n}) \\
&\leq &(1-\beta _{n})d^{2}(x_{n},p)+\beta _{n}d^{2}(x_{n},p) \\
&&-(1-\beta _{n})\beta _{n}d^{2}(x_{n},Tx_{n}) \\
&\leq &d^{2}(x_{n},p)-(1-\beta _{n})\beta _{n}d^{2}(x_{n},Tx_{n}) \\
&\leq &d^{2}(x_{n},p)
\end{eqnarray*}%
and 
\begin{eqnarray*}
d^{2}(x_{n+1},p) &=&d^{2}(P_{K}((1-\alpha _{n})y_{n}\oplus \alpha
_{n}u_{n}),P_{K}(p)) \\
&\leq &d^{2}((1-\alpha _{n})y_{n}\oplus \alpha _{n}u_{n}),p) \\
&\leq &(1-\alpha _{n})d^{2}(y_{n},p)+\alpha _{n}d^{2}(u_{n},p) \\
&&-(1-\alpha _{n})\alpha _{n}d^{2}(y_{n},u_{n}) \\
&\leq &(1-\alpha _{n})d^{2}(y_{n},p)+\alpha _{n}d^{2}(u_{n},Tp) \\
&&-(1-\alpha _{n})\alpha _{n}d^{2}(y_{n},Ty_{n}) \\
&\leq &(1-\alpha _{n})d^{2}(y_{n},p)+\alpha _{n}H^{2}(Ty_{n},Tp) \\
&&-(1-\alpha _{n})\alpha _{n}d^{2}(y_{n},Ty_{n}) \\
&\leq &(1-\alpha _{n})d^{2}(y_{n},p)+\alpha _{n}d^{2}(y_{n},p) \\
&&-(1-\alpha _{n})\alpha _{n}d^{2}(y_{n},Ty_{n}) \\
&\leq &d^{2}(y_{n},p)-(1-\alpha _{n})\alpha _{n}d^{2}(y_{n},Ty_{n}) \\
&\leq &d^{2}(y_{n},p) \\
&\leq &d^{2}(x_{n},p).
\end{eqnarray*}%
Here we have $d^{2}(x_{n+1},p)\leq d^{2}(x_{n},p)$ implies that $%
\lim_{n\rightarrow \infty }d(x_{n},p)$ exists, it is bounded,and $%
d(x_{n+1},p)\leq d(y_{n},p)\leq d(x_{n},p)$ implies $\lim_{n\rightarrow
\infty }[d(x_{n},p)-d(y_{n},p)]=0$. Since $\beta _{n}(1-\beta
_{n})d^{2}(Tx_{n},x_{n}))\leq d^{2}(x_{n},p)-d^{2}(y_{n},p),$by assumption
we have that $\lim_{n\rightarrow \infty }d^{2}(Tx_{n},x_{n})=0,$so $%
\lim_{n\rightarrow \infty }d(Tx_{n},x_{n})=0$
\end{proof}

\begin{theorem}
\label{teo4.5}If $T:K\rightarrow KC(X)$ is a nonexpansive mapping with $%
F(T)\neq \emptyset $ and $Tp=\{p\}$ for all $p\in F(T)$ and $\{x_{n}\}$ is a
sequence in $K$ \ defined by (\ref{1.1}) with $\lim \inf_{n\rightarrow
\infty }\beta _{n}(1-\beta _{n})>0$ then $\{x_{n}\}$ is $\Delta -$convergent
to $p\in $ $F(T)$ where $(p,p)$ is a solution of the problem (\ref{1.2})
\end{theorem}

\begin{proof}
Since we have $\lim_{n\rightarrow \infty }d(x_{n},Tx_{n})=0$, $%
\{d(x_{n},p)\} $ converges for all $p\in F(T)$ and $\{x_{n}\}$ is bounded by
Theorem \ref{teo4.4} then it follows from Lemma \ref{(lemma4.3)} that $%
\omega _{w}(x_{n})\subseteq F(T)$ and $\omega _{w}(x_{n})$ include exactly
one point $p\in $ $F(T)$ where $(p,p)$ is a solution of the problem \ref{1.2}
\end{proof}

\begin{theorem}
Let $K$ be also compact and $T:K\rightarrow C(X)$ be a nonexpansive mapping
with $F(T)\neq \emptyset $ and $Tp=\{p\}$ for all $p\in F(T)$. If $\{x_{n}\}$
is a sequence in $K$ \ defined by (\ref{1.1}) with $\lim \inf_{n\rightarrow
\infty }\beta _{n}(1-\beta _{n})>0$ then $\{x_{n}\}$ strongly converges to $%
q\in F(T).$where $(q,q)$ is a solution of the problem (\ref{1.2})
\end{theorem}

\begin{proof}
By Theorem \ref{teo4.4}, we have that $\lim_{n\rightarrow \infty
}d(Tx_{n},x_{n})=0$ and $\lim_{n\rightarrow \infty }d(x_{n},p)$ exists for
all $p\in F(T)$\ Since $K$ is compact there is a convergent subsequence $%
\{x_{n_{i}}\}$ of $\{x_{n}\},$ say $\lim_{i\rightarrow \infty
}x_{_{n_{i}}}=q.$ Then we have

\begin{equation*}
d(q,Tq)\leq
d(q,x_{_{n_{i}}})+d(x_{_{n_{i}}},Tx_{_{n_{i}}})+H(Tx_{_{n_{i}}},Tq)
\end{equation*}%
and taking limit on $i,$continuity of $T$ implies that $\ q\in Tq.$
\end{proof}

\section{Common Solution of System of Variational Inequalities}

Let $X$ be a $CAT(0)$ space and $K_{i}\subset X$ be a nonempty, closed and
convex subsets with $\bigcap\limits_{i=1}^{N}K_{i}\neq \emptyset $. If $%
T_{i}:K_{i}\rightarrow C(X)$ are mappings for $i=1...N$. then the system of
variational inequalities problem is

\begin{equation}
\text{Find }(u_{i},x)_{u_{i}\in T_{i}x}\text{ such that }\langle 
\overrightarrow{u_{i}x},\overrightarrow{xy}\rangle \geq 0\text{ for all }%
y\in K_{i},i=1,...,N  \label{4.1}
\end{equation}

It is obvious that for $N=1$ the problem is reduced the problem (\ref{1.2}).
The importance of studying the problem (\ref{4.1}) is underlying on fact
that it is unification most of the problems; for example taking if we take $%
T_{i}=0$ for all $i=1,...,N$ \ the reduce the problem (\ref{4.1})\ to convex
feasibility problem, 
\begin{equation*}
\text{Find }x\in K=\bigcap\limits_{i=1}^{N}K_{i}
\end{equation*}%
or if every $T_{i}$ is self operator and $K=\bigcap\limits_{i=1}^{N}F(T_{i})$
then it turn to common fixed point problem. We will show that the algorithm
defined by (\ref{4.2}) is convergent to common fixed point of family of
non-self multivalued nonexpansive mappings $\{T_{i}\}_{i=1}^{N}$ which is
also a solution of system of variational inequalities problem (\ref{4.1})
Let $K=\bigcap\limits_{i=1}^{N}K_{i}\neq \emptyset $ and $x_{1}\in K$. then
for any $n\geq 0$, the modified proximal multivalued Picard-S iteration is
defined by 
\begin{eqnarray}
x_{n+1} &=&P_{K}(\bigoplus\limits_{i=1}^{N}\lambda _{n,i}u_{n,i}),  \notag
\\
y_{n} &=&P_{K}(\bigoplus\limits_{i=1}^{N}\alpha _{n,i}w_{n,i}\oplus
\bigoplus\limits_{i=1}^{N}\beta _{n,i}v_{n,i}),  \notag \\
z_{n} &=&P_{K}(\gamma _{n,0}x_{n}\oplus \bigoplus\limits_{i=1}^{N}\gamma
_{n,i}w_{n,i})  \label{4.2}
\end{eqnarray}%
where $u_{n,i}\in T_{i}y_{n},w_{n,i}\in T_{i}x_{n},v_{n,i}\in
T_{i}z_{n},\{\lambda _{n,i}\},\{\alpha _{n,i}\},\{\beta _{n,i}\}$ and $%
\{\gamma _{n,i}\}$ are the sequences satisfies $\sum\limits_{i=1}^{N}%
\lambda _{n,i}=1,\sum\limits_{i=1}^{N}(\alpha _{n,i}+\beta
_{n,i})=1,\sum\limits_{i=0}^{N}\gamma _{n,i}=1$ in $[b,c]$ for some $b,c\in
(0,1)$

\begin{lemma}
\cite{les} Let $(X,d,W)$ be a uniformly convex hyperbolic space with modulus
of uniform convexity $\delta $. For any $r>0,$ $\epsilon \in (0,2),$ $%
\lambda \in \lbrack 0,1]$ and $a,x,y\in $ $X$, if $d(x,a)\leq r$, $%
d(y,a)\leq r$ and $d(x,y)\geq \in r$ then $d((1-\lambda )x\oplus \lambda
y,z)\leq (1-2\lambda (1-\lambda )\delta (r,\in ))r.$
\end{lemma}

\begin{proposition}
\cite{lao} Assume that $X$ is a $CAT(0)$ space. Then $X$ is uniformly convex
and
\end{proposition}

\begin{equation*}
\delta (r,\in )=\frac{\in ^{2}}{8}
\end{equation*}%
is a modulus of uniform convexity.

\begin{lemma}
\label{lemma7}\cite{dogh} Let $(X,d)$ be a complete $CAT(0)$ space, $%
\{x_{1},x_{2},..x_{n}\}\subset X$\ and $\{\lambda _{1},\lambda
_{2},..,\lambda _{n}\}\subset \lbrack 0,1]$\ with $\sum\limits_{i=1}^{n}%
\lambda _{i}=1.$Then $d(\bigoplus\limits_{i=1}^{n}\lambda _{i}x_{i},z)\leq
\sum\limits_{i=1}^{n}\lambda _{i}d(x_{i},z)$ for every $z\in X.$
\end{lemma}

\begin{lemma}
\label{lemma8}\cite{lao} Let $X$ be a complete $CAT(0)$ space with modulus
of convexity $\delta (r,\in )$ and let $x\in E$. Suppose that $\delta (r,\in
)$ increases with $r$ (for a fixed $\in $ ) and suppose $\{t_{n}\}$ is a
sequence in $[b,c]$ for some $b,c\in (0,1)$, $\{x_{n}\}$ and $\{y_{n}\}$ are
the sequences in $X$ such that $\limsup\nolimits_{n\rightarrow \infty
}d(x_{n},x)\leq r,$ $\limsup\nolimits_{n\rightarrow \infty }d(y_{n},x)\leq r$
and $\lim\nolimits_{n\rightarrow \infty }d((1-t_{n})x_{n}\oplus
t_{n}y_{n},x)=r$ for some $r\geq 0$. Then $\lim_{n\rightarrow \infty
}d(x_{n},y_{n})=0.$

The following Lemma is very important to our results.
\end{lemma}

\begin{lemma}
\label{lemma9}Let $X$ be a complete $CAT(0)$ space with modulus of convexity 
$\delta (r,\in )$ and let $x\in X$. Suppose that $\delta (r,\in )$ increases
with $r$ (for a fixed $\in $ ) and suppose $\{t_{n,i}\}$ with $%
\sum\limits_{i=1}^{N}t_{n,i}=1$ is a sequence in $[b,c]$ for some $b,c\in
(0,1)$, $\{x_{n,i}\}_{n=1}^{\infty }$ are the sequences for $i\in
\{1,2,..,N\}$ in $X$ such that $\limsup\nolimits_{n\rightarrow \infty
}d(x_{n,i},x)\leq r$ and $\lim\nolimits_{n\rightarrow \infty
}d(\bigoplus\limits_{i=1}^{N}t_{n,i}x_{n,i},x)=r$ for some $r\geq 0$. Then $%
\lim\nolimits_{n\rightarrow \infty }d(x_{n,k},x_{n,l})=0$ for $k,l\in
\{1,2,..,N\}.$
\end{lemma}

\begin{proof}
If $r=0$ then it is obvious let $r>0.$Since $\limsup\nolimits_{n\rightarrow
\infty }d(x_{n,i},x)\leq r$ for each $i=1,2,..N$, then, by Lemma \ref{lemma7}%
, for every $m=1,2,..,N$,%
\begin{eqnarray*}
\lim_{n\rightarrow \infty }d(\bigoplus\limits_{\substack{ i=1,  \\ i\neq m}}%
^{N}\frac{t_{n,i}}{1-t_{n,m}}x_{n,i},x) &\leq &\lim_{n\rightarrow \infty
}\sum\limits_{\substack{ i=1,  \\ i\neq m}}^{n}\frac{t_{n,i}}{1-t_{n,m}}%
d(x_{n,i},x) \\
&\leq &\lim_{n\rightarrow \infty }\sum\limits_{\substack{ i=1,  \\ i\neq m}}%
^{n}\frac{t_{n,i}}{1-t_{n,m}}(\limsup\nolimits_{n\rightarrow \infty
}d(x_{n,i},x)) \\
&\leq &\lim_{n\rightarrow \infty }\sum\limits_{\substack{ i=1,  \\ i\neq m}}%
^{n}\frac{t_{n,i}}{1-t_{n,m}}r=r.
\end{eqnarray*}%
Let assume that $d(x_{n,k},x_{n,l})\nrightarrow 0$ for fixed $k,l\in
\{1,2,..,N\}$ with $k\neq l$ \ then there is subsequence denoted by (without
loss of generality) $\{x_{n,k}\}$ and $\{x_{n,l}\}$ such that $%
inf_{n}d(x_{n,k},x_{n,l})>0.$Since 
\begin{eqnarray*}
d(\bigoplus\limits_{i=1}^{N}t_{n,i}x_{n,i},x_{n,m})
&=&d((1-t_{n,m})[\bigoplus\limits_{\substack{ i=1  \\ i\neq m}}^{N}\frac{%
t_{n,i}}{1-t_{n,m}}x_{n,i}]\oplus t_{n,m}x_{n,m},x_{n,m}) \\
&\leq &(1-t_{n,m})[d(\bigoplus\limits_{\substack{ i=1,  \\ i\neq m}}^{N}%
\frac{t_{n,i}}{1-t_{n,m}}x_{n,i},x_{n,m})+t_{n,m}d(x_{n,m},x_{n,m}) \\
&=&(1-t_{n,m})[d(\bigoplus\limits_{\substack{ i=1,  \\ i\neq m}}^{N}\frac{%
t_{n,i}}{1-t_{n,m}}x_{n,i},x_{n,m})
\end{eqnarray*}%
then 
\begin{eqnarray*}
0 &<&d(x_{n,k},x_{n,l}) \\
&\leq
&d(\bigoplus\limits_{i=1}^{N}t_{n,i}x_{n,i},x_{n,k})+d(\bigoplus%
\limits_{i=1}^{N}t_{n,i}x_{i},x_{n,l}) \\
&\leq &(1-t_{n,k})d(\bigoplus\limits_{\substack{ i=1,  \\ i\neq k}}^{N}%
\frac{t_{n,i}}{1-t_{n,k}}x_{n,i},x_{n,k})+(1-t_{n,l})d(\bigoplus\limits 
_{\substack{ i=1,  \\ i\neq l}}^{N}\frac{t_{n,i}}{1-t_{n,l}}x_{n,i},x_{n,l})
\end{eqnarray*}%
and since $t_{n,k},t_{n,l}\in \lbrack b,c]$ and by positivity of $d$, $%
d(\bigoplus\limits_{i=1,i\neq k}^{N}\frac{t_{n,i}}{1-t_{n,k}}%
x_{n,i},x_{n,k})\nrightarrow 0.$ therefore there is subsequence again
denoted by $\{x_{n,k}\}$ for some $k=1,2,..N$ such that $d(\bigoplus%
\limits_{i=1}^{N}\frac{t_{n,i}}{1-t_{n,k}}x_{n,i},x_{n,k})>0$ so $%
d(x_{n,k},x)\leq r,d(\bigoplus\limits_{i=1,i\neq k}^{N}\frac{t_{n,i}}{%
1-t_{n,k}}x_{n,k},x)\leq r$ and $\lim\nolimits_{n\rightarrow \infty
}d(\bigoplus\limits_{i=1}^{N}t_{n,i}x_{n,i},x)=\lim\nolimits_{n\rightarrow
\infty }d((1-t_{n,m})[\bigoplus\limits_{i=1,i\neq k}^{N}\frac{t_{n,i}}{%
1-t_{n,k}}x_{n,i}]\oplus t_{n,m}x_{n,m},x)=r$ hence we can apply Lemma \ref%
{lemma8}.
\end{proof}

From this point, it is assumed that $X$ is a complete $CAT(0)$ and $%
K=\bigcap\limits_{i=1}^{N}K_{i}$ is a nonempty, closed and convex subset of $%
X$ where $K_{i}\subset X$ be a nonempty, closed and convex subsets with $%
K=\bigcap\limits_{i=1}^{N}K_{i}\neq \emptyset $ for all $i=1,2,...,N$.

\begin{lemma}
\label{lemma10}Let $\{T_{i}\}_{i=1}^{N}$ be multivalued nonexpansive
mappings \ from $K$ to $CC(X)$ with $F=\bigcap\limits_{i=1}^{N}F(T_{i})\neq
\emptyset ,$ $T_{i}p=\{p\}$ for all $p\in F$. If\ $\{x_{n}\}$ is the
sequence defined by (\ref{4.2}) then $\{x_{n}\}$ is bounded and $%
\lim_{n\rightarrow \infty }d(x_{n},p)$ exist for all $p\in F$.
\end{lemma}

\begin{proof}
Let $p\in F.$Then from definition of $\{x_{n}\},$%
\begin{eqnarray*}
d(x_{n+1},p) &=&d(P_{K}(\bigoplus\limits_{i=1}^{N}\lambda _{n,i}u_{n,i}),p)
\\
&\leq &\sum\limits_{i=1}^{N}\lambda _{n,i}d(u_{n,i,}p) \\
&\leq &\sum\limits_{i=1}^{N}\lambda _{n,i}d(u_{n,i,}T_{i}p) \\
&\leq &\sum\limits_{i=1}^{N}\lambda _{n,i}H(T_{i}y_{n},T_{i}p) \\
&\leq &\sum\limits_{i=1}^{N}\lambda _{n,i}d(y_{n},p) \\
&=&d(y_{n},p)
\end{eqnarray*}%
and%
\begin{eqnarray*}
d(y_{n},p) &=&d(\bigoplus\limits_{i=1}^{N}\alpha _{n,i}w_{n,i}\oplus
\bigoplus\limits_{i=1}^{N}\beta _{n,i}v_{n,i},,p) \\
&\leq &\sum\limits_{i=1}^{N}\alpha
_{n,i}d(w_{n,i},p)+\sum\limits_{i=1}^{N}\beta _{n,i}d(v_{n,i},p) \\
&\leq &\sum\limits_{i=1}^{N}\alpha
_{n,i}d(w_{n,i},T_{i}p)+\sum\limits_{i=1}^{N}\beta _{n,i}d(v_{n,i},T_{i}p)
\\
&\leq &\sum\limits_{i=1}^{N}\alpha
_{n,i}H(T_{i}x_{n},T_{i}p)+\sum\limits_{i=1}^{N}\beta
_{n,i}H(T_{i}z_{n},T_{i}p) \\
&\leq &\sum\limits_{i=1}^{N}\alpha
_{n,i}d(x_{n},p)+\sum\limits_{i=1}^{N}\beta _{n,i}d(z_{n},p)
\end{eqnarray*}%
and%
\begin{eqnarray*}
d(z_{n},p) &=&d(\gamma _{n,0}x_{n}\oplus \bigoplus\limits_{i=1}^{N}\gamma
_{n,i}w_{n,i},p) \\
&\leq &\gamma _{n,0}d(x_{n},p)+\sum\limits_{i=1}^{N}\gamma
_{n,i}d(w_{n,i},p) \\
&\leq &\gamma _{n,0}d(x_{n},p)+\sum\limits_{i=1}^{N}\gamma
_{n,i}d(w_{n,i},T_{i}p) \\
&\leq &\gamma _{n,0}d(x_{n},p)+\sum\limits_{i=1}^{N}\gamma
_{n,i}H(T_{i}x_{n},T_{i}p) \\
&\leq &\gamma _{n,0}d(x_{n},p)+\sum\limits_{i=1}^{N}\gamma _{n,i}d(x_{n},p)
\\
&=&d(x_{n},p).
\end{eqnarray*}%
\qquad Hence $d(y_{n},p)\leq $ $d(x_{n},p),d(z_{n},p)\leq d(x_{n},p)$ and $%
d(x_{n+1},p)\leq d(x_{n},p)$ and so $\lim_{n\rightarrow \infty }d(x_{n},p)$
exists and $\{x_{n}\}$ is bounded sequence.
\end{proof}

\begin{lemma}
\label{lemma11}Let $\{T_{i}\}_{i=1}^{N}$ be multivalued nonexpansive
mappings from $K$ to $C(X)$ with $F=\bigcap\limits_{i=1}^{N}F(T_{i})\neq
\emptyset ,$ $T_{i}p=\{p\}$ for all $p\in F$. If\ $\{x_{n}\}$ is the
sequence defined by (\ref{4.2}) then $\lim_{n\rightarrow \infty
}d(x_{n},T_{i}x_{n})$ exist for all $i=1,2..N$.

\begin{proof}
Let $p\in F.$From the Lemma \ref{lemma10} $\lim_{n\rightarrow \infty
}d(x_{n},p)$ exist and $\{x_{n}\}$ is bounded sequence. so let $%
\lim_{n\rightarrow \infty }d(x_{n},p)=c$. Since $d(y_{n},p)\leq $ $%
d(x_{n},p) $ and $d(u_{n,i},p)\leq d(y_{n},p),\limsup_{n\rightarrow \infty
}d(y_{n},p)\leq c$ and $\limsup_{n\rightarrow \infty }d(u_{n,i},p)\leq c$
and again from Lemma \ref{lemma10} similarly $\limsup_{n\rightarrow \infty
}d(z_{n},p)\leq c$ and $\limsup_{n\rightarrow \infty }d(v_{n,i},p)\leq c$
and $\limsup_{n\rightarrow \infty }d(x_{n},p)\leq c$ and $%
\limsup_{n\rightarrow \infty }d(w_{n,i},p)\leq c.$Moreover we have%
\begin{eqnarray*}
c &=&\lim_{n\rightarrow \infty }d(x_{n+1},p) \\
&=&\lim_{n\rightarrow \infty }d(\bigoplus\limits_{i=1}^{N}\lambda
_{n,i}u_{n,i},p)\text{ } \\
\text{ } &\leq &\lim_{n\rightarrow \infty }\sum\limits_{i=1}^{N}\lambda
_{n,i}d(u_{n,i},p) \\
&\leq &\lim_{n\rightarrow \infty }\sum\limits_{i=1}^{N}\lambda
_{n,i}\limsup_{n\rightarrow \infty }d(u_{n,i},p) \\
&\leq &\lim_{n\rightarrow \infty }\sum\limits_{i=1}^{N}\lambda _{n,i}c\leq c
\end{eqnarray*}%
implies $\lim_{n\rightarrow \infty }d(\bigoplus\nolimits_{i=1}^{N}\lambda
_{n,i}u_{n,i},p)=c$. We find that $\lim_{n\rightarrow \infty
}d(u_{n,i},u_{n,j})=0$ for all $i,j=1,2,..,N$ by Lemma \ref{lemma9}. Then 
\begin{eqnarray*}
d(x_{n+1},p) &=&d(\bigoplus\limits_{i=1}^{N}\lambda _{n,i}u_{n,i,}p) \\
&\leq &\sum\limits_{i=1}^{N}\lambda _{n,i}d(u_{n,i,}p) \\
&\leq &\sum\limits_{i=1}^{N}\lambda _{n,i}d(u_{n,i,}p) \\
&\leq &\sum\limits_{i=1}^{N}\lambda _{n,i}d(u_{n,i,}p) \\
&\leq &\sum\limits_{i=1}^{N}\lambda
_{n,i}[(d(u_{n,i,}u_{n,m})+d(u_{n,m},p))] \\
&\leq &d(u_{n,m},p)+\sum\limits_{i=1}^{N}\lambda _{n,i}d(u_{n,i,}u_{n,m})
\end{eqnarray*}%
and then $\liminf_{n\rightarrow \infty }d(u_{n,m},p)\geq c$ for all $%
m=1,2,..,N$. Since $\limsup_{n\rightarrow \infty }d(u_{n,i,},p)\leq c$ and $%
d(u_{n,i,}p)\leq d(y_{n},p)$ thus we have $\lim_{n\rightarrow \infty
}d(u_{n,i},p)=c$ and $\lim_{n\rightarrow \infty }d(y_{n},p)$=$c$.%
\begin{eqnarray*}
c &=&\lim_{n\rightarrow \infty }d(y_{n},p) \\
&=&\lim_{n\rightarrow \infty }d(\bigoplus\limits_{i=1}^{N}\alpha
_{n,i}w_{n,i}\oplus \bigoplus\limits_{i=1}^{N}\beta _{n,i}v_{n,i},p) \\
&\leq &\lim_{n\rightarrow \infty }[\sum\limits_{i=1}^{N}\alpha
_{n,i}\limsup_{n\rightarrow \infty }d(x_{n},p)+\sum\limits_{i=1}^{N}\beta
_{n,i}\limsup_{n\rightarrow \infty }d(w_{n,i},p)] \\
&\leq &\lim_{n\rightarrow \infty }[\sum\limits_{i=1}^{N}\alpha
_{n,i}c+\sum\limits_{i=1}^{N}\beta _{n,i}c]\leq c
\end{eqnarray*}%
which implies that $\lim_{n\rightarrow \infty
}d(\bigoplus\nolimits_{i=1}^{N}\alpha _{n,i}w_{n,i}\oplus
\bigoplus\nolimits_{i=1}^{N}\beta _{n,i}v_{n,i},p)=c$. Also we have $%
\lim_{n\rightarrow \infty }d(v_{n,i},v_{n,j})=0$, $\lim_{n\rightarrow \infty
}d(v_{n,i},w_{n,j})=0$ and $\lim_{n\rightarrow \infty }d(w_{n,i},w_{n,j})=0$
for all $i,j=1,...,N$ by Lemma \ref{lemma11}. Then%
\begin{eqnarray*}
d(y_{n},p) &=&d(\bigoplus\limits_{i=1}^{N}\alpha _{n,i}w_{n,i}\oplus
\bigoplus\limits_{i=1}^{N}\beta _{n,i}v_{n,i},p) \\
&\leq &\sum\limits_{i=1}^{N}\alpha
_{n,i}d(w_{n,i},p)+\sum\limits_{i=1}^{N}\beta _{n,i}d(v_{n,i},p) \\
&\leq &\sum\limits_{i=1}^{N}\alpha
_{n,i}[d(w_{n,i},v_{n,m})+d(v_{n,m},p)]+\sum\limits_{i=1}^{N}\beta
_{n,i}d(v_{n,i},p) \\
&\leq &\sum\limits_{i=1}^{N}\alpha
_{n,i}d(w_{n,i},v_{n,m})+(1-\sum\limits_{i=1}^{N}\beta
_{n,i})d(v_{n,m},p)+\sum\limits_{i=1}^{N}\beta _{n,i}d(v_{n,i},p) \\
&=&\sum\limits_{i=1}^{N}\alpha
_{n,i}d(w_{n,i},v_{n,m})+d(v_{n,m},p)+\sum\limits_{i=1}^{N}\beta
_{n,i}[d(v_{n,i},p)-d(v_{n,m},p)] \\
&\leq &\sum\limits_{i=1}^{N}\alpha _{n,i}d(w_{n,i},v_{n,m})+d(v_{n,m},p) \\
&&+\sum\limits_{i=1}^{N}\beta
_{n,i}[d(v_{n,i},v_{n,m})+(d(v_{n,m},p))-d(v_{n,m},p)] \\
&\leq &\sum\limits_{i=1}^{N}\alpha
_{n,i}d(w_{n,i},v_{n,m})+d(v_{n,m},p)+\sum\limits_{i=1}^{N}\beta
_{n,i}d(v_{n,i},v_{n,m})
\end{eqnarray*}%
and since $\lim_{n\rightarrow \infty }d(v_{n,i},w_{n,j})=0$ and $%
\lim_{n\rightarrow \infty }d(w_{n,i},w_{n,j})=0$ for all $i,j=1,...,N$ then $%
\liminf_{n\rightarrow \infty }d(v_{n,m},p)\geq c$ for all $m=1,2,..,N$ and
since $\limsup_{n\rightarrow \infty }d(v_{n,i,},p)\leq c$ and $%
d(v_{n,i,}p)\leq d(z_{n},p)$ thus we have $\lim_{n\rightarrow \infty
}d(v_{n,i},p)=c.$and $\lim_{n\rightarrow \infty }d(z_{n},p)$=$c$. Finally 
\begin{eqnarray*}
c &=&\lim_{n\rightarrow \infty }d(z_{n},p) \\
&=&\lim_{n\rightarrow \infty }[d(\gamma _{n,0}x_{n}\oplus
\bigoplus\limits_{i=1}^{N}\gamma _{n,i}w_{n,i},p) \\
&\leq &\lim_{n\rightarrow \infty }[\gamma _{n,0}\limsup_{n\rightarrow \infty
}d(x_{n},p)+\sum\limits_{i=1}^{N}\gamma _{n,i}\limsup_{n\rightarrow \infty
}d(w_{n,i},p)] \\
&\leq &\lim_{n\rightarrow \infty }[\gamma
_{n,0}c+\sum\limits_{i=1}^{N}\gamma _{n,i}c]\leq c
\end{eqnarray*}%
which implies that $\lim_{n\rightarrow \infty }[d(\gamma _{n,0}x_{n}\oplus
\bigoplus\nolimits_{i=1}^{N}\gamma _{n,i}w_{n,i},p)=c$ and since $%
\limsup_{n\rightarrow \infty }d(x_{n},p)\leq c$ and $\limsup_{n\rightarrow
\infty }d(w_{n,i},p)\leq c$ we find that $\lim_{n\rightarrow \infty
}d(x_{n},w_{n,i})=0$ and $\lim_{n\rightarrow \infty }d(w_{n,i},w_{n,j})=0$
for all $i,j$ by Lemma \ref{lemma9}..Hence $d(x_{n},T_{i}x_{n})\leq
d(x_{n},w_{n,i})$ for all $i=1,2..N$ and $\lim_{n\rightarrow \infty
}d(x_{n},T_{i}x_{n})$ $=0$
\end{proof}
\end{lemma}

\begin{theorem}
Let $\{T_{i}\}_{i=1}^{N}$ be multivalued nonexpansive mappings from $K$ to $%
KC(X)$ with $F=\bigcap\nolimits_{i=1}^{N}F(T_{i})\neq \emptyset ,$ $%
T_{i}p=\{p\}$ for all $p\in F$ . \ Then a sequence\ $\{x_{n}\}$ defined by (%
\ref{4.2}) $\Delta $-converges to $p\in F$ where $(p,p)$ is a common
solution of the problem (\ref{4.1}).
\end{theorem}

\begin{proof}
It follows from Lemma \ref{lemma10} and Lemma \ref{lemma11} that \ $%
\lim_{n\rightarrow \infty }d(x_{n},T_{i}x_{n})=0$ for all $i\in \{1,2..N\}$, 
$\lim_{n\rightarrow \infty }d(x_{n},p)$ exists for all $p\in F$. Let $\omega
_{w}(x_{n}):=\cup A(\{u_{n}\})$ where union take on all subsequence $%
\{u_{n}\}$ of $\{x_{n}\}.$ To show that $\Delta $-convergence of $\{x_{n}\}$
it is enough to show that $\omega _{w}(x_{n})\subseteq F$ and $\omega
_{w}(x_{n})$ contains single point. First of all $\omega _{w}(x_{n})\subset
K $ by Lemma \ref{(lemma2.10)}. Let take$\ u\in \omega _{w}(x_{n})$, then
there exist subsequence $\{u_{n}\}$ of $\{x_{n}\}$ such that $\
A\{u_{n}\}=\{u\}.$By Lemma \ref{(lemma2.10)} and Lemma \ref{(lemma2.11)}
there exist a subsequence $(v_{n})$ of $\{u_{n}\}$ which $\Delta -$%
convergent to $v$. Let fix $i\in \{1,2...N\},$Since $T_{i}v$ is compact,
then for each $n\geq 1$ we can pick up $z_{n,i}\in T_{i}v$ satisfies $%
d(v_{n},z_{n,i})=d(v_{n},T_{i}v)$ and compactness of $T_{i}v$ implies there
exist a convergent subsequence $\{z_{n_{k},i}\}$ of $\{z_{n,i}\}$. Let $%
z_{n_{k},i}\rightarrow w_{i}\in T_{i}v$. Since $T_{i}$ is nonexpansive map
we have;%
\begin{eqnarray*}
d(v_{n_{k}},z_{n_{k},i}) &=&d(v_{n_{k}},T_{i}v)\leq
d(v_{n_{k}},T_{i}v_{n_{k}})+H(T_{i}v_{n_{k}},T_{i}v) \\
&\leq &d(v_{n_{k}},T_{i}v_{n_{k}})+d(v_{n_{k}},v)
\end{eqnarray*}%
Hence we have%
\begin{eqnarray*}
d(v_{n_{k}}w_{i}) &\leq &d(v_{n_{k}},z_{n_{k},i})+d(z_{n_{k},i},w_{i}) \\
&\leq &d(v_{n_{k}},T_{i}v_{n_{k}})+d(v_{n_{k}},v)+d(z_{n_{k},i},w_{i})
\end{eqnarray*}%
which implies%
\begin{equation*}
\limsup_{n\rightarrow \infty }d(v_{n_{k}},w_{i})\leq \limsup_{n\rightarrow
\infty }d(v_{n_{k}}v)
\end{equation*}%
Hence by uniqueness of asymptotic centers, we have $w_{i}=v\in T_{i}v.$
Since $i$ was arbitrary we have $v\in $ $F=\bigcap\limits_{i=1}^{N}F(T_{i})$
so $\lim\nolimits_{n\rightarrow \infty }d(x_{n},v)$ exist by Lemma \ref%
{lemma10} which implies $u=v\in F$ by Lemma \ref{(lemma2.11)}.Thus we have $%
\omega _{w}(x_{n})\subseteq F.$ If we take subsequence $\{u_{n}\}$ of $%
\{x_{n}\}$ with $A\{u_{n}\}=\{u\}$ and $A\{x_{n}\}=\{x\}$ then, since $\
u\in \omega _{w}(x_{n})\subseteq F$ and $\lim\nolimits_{n\rightarrow \infty
}d(x_{n},v)$ exist, we have $u=x$ by Lemma \ref{(lemma2.11)}.
\end{proof}

\begin{theorem}
If $K$ is also compact $\{T_{i}\}_{i=1}^{N}$ are multivalued nonexpansive
mappings from $K$ to $C(X)$ with $F=\bigcap\nolimits_{i=1}^{N}F(T_{i})\neq
\emptyset ,$ $T_{i}p=\{p\}$ for all $p\in F$ then the sequence\ $\{x_{n}\}$
defined by (\ref{4.2}) strongly strongly converges to $p\in F$ where $(p,p)$
is a common solution of the problem (\ref{4.1})
\end{theorem}

\begin{proof}
By Lemma \ref{lemma10} and Lemma \ref{lemma11},we have that\ $%
\lim_{n\rightarrow \infty }d(x_{n},T_{i}x_{n})=0$ for all $i\in \{1,2..N\}$, 
$\lim_{n\rightarrow \infty }d(x_{n},p)$ exists for all $p\in F$ Since $K$ is
compact there is a convergent subsequence $\{x_{n_{k}}\}$ of $\{x_{n}\},$
say $\lim_{i\rightarrow \infty }x_{_{n_{k}}}=q.$ Then for all $i\in
\{1,2..N\},$ we have

\begin{eqnarray*}
d(q,T_{i}q) &\leq
&d(q,x_{_{n_{k}}})+d(x_{_{n_{k}}},T_{i}x_{_{n_{k}}})+H(T_{i}x_{_{n_{k}}},T_{i}q)
\\
&\leq &d(q,x_{_{n_{k}}})+d(x_{_{n_{k}}},T_{i}x_{_{n_{k}}})+d(x_{_{n_{k}}},q)
\end{eqnarray*}%
and taking limit on $k,$implies that $\ q\in T_{i}q$ for all $i\in
\{1,2..N\} $. Hence $p\in F$
\end{proof}


\begin{thebibliography}{99}
\bibitem{nadler} S. Nadler (1969). Multi-valued contraction mappings.
Pacific Journal of Mathematics, 30(2):475-488.

\bibitem{kind1} D. Kinderlehrer, G. Stampacchia(1980). An Introduction to
Variational Inequalities and their Applications. \textit{Pure and Applied
Mathematics} Volume 88. Academic Press, New York

\bibitem{kind2} P. Hartman, G. Stampacchia(1966). On some non-linear
elliptic differential-functional equations.. Acta Mathematica, 115:153--188

\bibitem{or1} S-s. Chang, BS Lee, Y-O Chen(1995). Variational inequalities
for monotone operators in nonreflexive Banach spaces. \textit{Applied
Mathematics Letters}, 8(6):29--34.

\bibitem{or2} N-J Huang, Y-P Fang (2003). Fixed point theorems and a new
system of multivalued generalized order complementarity problems. \textit{%
Positivity}, 7(3):257--265. 10.1023/A:1026222030596

\bibitem{or3} G. Isac(1992). Complementarity Problems. Lecture Notes in
Mathematics,. Volume 1528. Springer, Berlin, Germany

\bibitem{or4} P. Junlouchai, S. Plubtieng (2011). Existence of solutions for
generalized variational inequality problems in Banach spaces. \textit{%
Nonlinear Analysis: Theory Methods \& Applications}, 74(3):999--1004.
10.1016/j.na.2010.09.058

\bibitem{or5} RU Verma (1999). On a new system of nonlinear variational
inequalities and associated iterative algorithms. \textit{Mathematical
Sciences Research Hot-Line}, 3(8):65--68.

\bibitem{dogh} S. Dhompongsa, A. Kaewkhao, and B. Panyanak (2012). On Kirk's
strong convergence theorem for multivalued nonexpansive mappings on CAT (0)
spaces. \textit{Nonlinear Analysis: Theory, Methods \& Applications} 75.2:
459-468

\bibitem{dogh2} S. Dhompongsa, B. Panyanak(2008). On 4-convergence theorems
in CAT(0) spaces. \textit{Comput. Math. Appl}.56:2572-2579.

\bibitem{shim} T.Shimizu \& W.Takahashi (1996). Fixed points of multivalued
mappings in certain convex metric spaces. \textit{Topological Methods in
Nonlinear Analysis}, 8(1):197-203.

\bibitem{esp} Esp\'{\i}nola, Rafa, and Aurora Fern\'{a}ndez-Le\'{o}n (2009).
CAT (k)-spaces, weak convergence and fixed points. \textit{Journal of
Mathematical Analysis and Applications} 353.1: 410-427.

\bibitem{brid} M. R. Bridson and A. Haefliger (1999). \textit{Metric Spaces
of Non-Positive Curvature}, Springer-Verlag, BerlinU " Heidelberg.

\bibitem{ohta} S. Ohta (2007). Convexities of metric spaces. \textit{Geom.
Dedic}. 125:225-250

\bibitem{kha} H. Khatibzadeh, \& S. Ranjbar(2015). A variational inequality
in complete CAT (0) spaces. \textit{Journal of Fixed Point Theory and
Applications,} 17(3), 557-574.

\bibitem{berg} I. D. Berg and I. G. Nikolaev(2008). Quasilinearization and
curvature of Aleksandrov spaces. \textit{Geom. Dedicata} 133 , 195--218.

\bibitem{gurs} F. Gursoy and V. Karakaya(2014) A Picard-S hybrid type
iteration method for solving a differential equation with retarded argument. 
\textit{ArXiv preprint}, arXiv:1403.2546

\bibitem{gurs2} F. Gursoy(2016). A Picard-S Iterative Method for
Approximating Fixed Point of Weak-Contraction Mappings. \textit{Filomat}.
30:2829--2845

\bibitem{les} L. Leustean(2007). Quadratic rate of asymptotic regularity for
CAT(0)-spaces. \textit{T J. Math. Anal. Appl}., 325:386-399

\bibitem{lao} W. Laowang and B. Panyanak(2010). Approximating fixed points
of nonexpansive non-self mappings in CAT(0) spaces. \textit{Fixed Point
Theory and Applications}, Article ID 367274, 11 \ 

\bibitem{rash} R. A Rashwan, and S. M. Altwqi (2015). On the convergence of
SP-iterative scheme for three multivalued nonexpansive mappings in CAT ($%
\kappa $) spaces. \textit{Palestine Journal of Mathematics, }Vol.
4(1):73--83\ 
\end{thebibliography}
\end{document}